\documentclass[12pt]{amsart}

\usepackage{amssymb,latexsym}

\usepackage{color}

\usepackage{enumerate}

\usepackage[T1]{fontenc}

 \usepackage[french,english]{babel}


\newtheorem{thm}{Theorem}[section]

\newtheorem{lem}[thm]{Lemma}

\theoremstyle{definition}

\numberwithin{equation}{section}

\renewcommand{\leq}{\leqslant}

\renewcommand{\geq}{\geqslant}

\newcommand{\p}{\mathcal{P}}

\newcommand{\re}{\textup{Re}}

\newcommand{\V}{\mathbf{v}}

\newcommand{\F}{\mathcal{F}}

\newcommand{\newabstract}[1]{%
  \par\bigskip
  \csname otherlanguage*\endcsname{#1}%
  \csname captions#1\endcsname
  \item[\hskip\labelsep\scshape\abstractname.]
}

\begin{document}

\baselineskip=17pt

\title[Large values of Dirichlet $L$-functions for characters of fixed order]{A note on large values of Dirichlet $L$-functions for characters of fixed order at $1/2<\sigma\leq 1$
}

\author{Youness Lamzouri}

\address{
Universit\'e de Lorraine, CNRS, IECL, 
F-54000 Nancy, France}

\email{youness.lamzouri@univ-lorraine.fr}


\begin{abstract} 
In this note, we use a simple argument to show the existence of large values of conjecturally sharp size for Dirichlet $L$-functions attached to primitive characters of fixed order at $\sigma\in (1/2, 1]$.
 More precisely, for every fixed integer $g\geq 2$ we prove the existence of a primitive character $\chi$ of order $g$ and conductor $Q\asymp x$ such that 
$
|L(1,\chi)|
\geq
e^\gamma\left(\log\log x+\log\log\log x-\log(2\log g)+o(1)\right).
$
We also show that for every fixed $1/2<\sigma<1$ there exists a primitive character $\chi$ of order $g$ and conductor $Q\asymp x$ such that
$
\log |L(\sigma,\chi)|
\geq
\left(C_g(\sigma)+o(1)\right)
(\log x)^{1-\sigma}(\log\log x)^{-\sigma},
$
for some explicit positive constant $C_g(\sigma).$
Previously, such bounds were known only conditionally on the Generalized Riemann Hypothesis, and even then only in the special cases $g=2$ and $g=3$. 
\end{abstract}


\maketitle


\section{Introduction} 

A central problem in analytic number theory is to understand the size of $L$-functions in the critical strip. This question is closely connected with several classical arithmetic problems, including class numbers of quadratic extensions of $\mathbb{Q}$, which are related to values of quadratic Dirichlet $L$-functions at $1$, and the non-vanishing of the Riemann zeta function $\zeta(s)$ on the line $\re(s)=1$, which is equivalent to the prime number theorem.

Let $1/2<\sigma<1$ be fixed. In \cite{Mo77} Montgomery proved the lower bound
\begin{equation}\label{Eq:Montgomery}
\max_{t\in [T,2T]}\log |\zeta(\sigma+it)|
\geq
(c(\sigma)+o(1))(\log T)^{1-\sigma}(\log\log T)^{-\sigma},
\end{equation}
for some positive constant $c(\sigma)$, and conjectured that this is best possible up to the value of the constant $c(\sigma)$.  On the line $\sigma=1$, Granville and Soundararajan \cite{GrSo03} proved that
$
\max_{t\in [T, 2T]}|\zeta(1+it)|
\geq
e^\gamma\left(\log_2 T+\log_3 T-\log_4 T+O(1)\right),$
where here and throughout, $\log_j$ denotes the $j$-fold iterated logarithm.
This was later improved by Aistleitner, Mahatab and Munsch \cite{AMM19}, who removed the $-\log_4 T$ term using a variant of the long resonator method, obtaining a bound which is conjectured to be best possible up to the constant in the $O(1)$ term.  Subsequently, Aistleitner, Mahatab, Munsch and Peyrot \cite{AMMP19} adapted the long resonator method to the family of Dirichlet characters modulo $q$. They proved that for every fixed $1/2<\sigma<1$ and all sufficiently large $q$, we have
\begin{equation}\label{Eq:Montgomery2}
\max_{\substack{\chi \bmod q\\ \chi\neq \chi_0}}\log |L(\sigma,\chi)|
\geq
(C(\sigma)+o(1))(\log q)^{1-\sigma}(\log_2 q)^{-\sigma},
\end{equation}
for some positive constant $C(\sigma)$, where $\chi_0$ is the principal character modulo $q$. In the case $\sigma=1$, they showed that 
$$
\max_{\substack{\chi \bmod q\\ \chi\neq \chi_0}}|L(1,\chi)|
\geq
e^\gamma\left(\log_2 q+\log_3 q-\log(2\log 2)-1+o(1)\right).
$$ 

The situation is more complicated for the family of quadratic characters, 
since prior to the present work the analogous bounds in this setting were only known conditionally on the Generalized Riemann Hypothesis (GRH). Indeed, Granville and Soundararajan \cite{GrSo03}
showed under GRH that there are many primes $q\leq x$ for which
$$
L(1,\chi_q)
\geq
e^\gamma\left(\log_2 x+\log_3 x-\log(2\log 2)+o(1)\right),
$$
where $\chi_q=(\frac{\cdot}{q})$ is the Legendre symbol modulo $q$. This idea was extended by the author \cite{La11} to the range $1/2<\sigma<1$, again assuming GRH, giving lower bounds analogous to \eqref{Eq:Montgomery} and \eqref{Eq:Montgomery2} for $L(\sigma,\chi_q)$ for prime discriminants $q\leq x.$

These conditional results  
are based on the following idea, which goes back to Littlewood and Chowla and was later used by Montgomery \cite{Mo71} in his work on the least quadratic non-residue. More precisely, one restricts to a thin family of prime discriminants for which the corresponding quadratic characters have value $1$ at the first few primes. This makes the initial Euler product large, and then quadratic reciprocity and the GRH are used to control the error terms coming from the off-diagonal contributions in such a thin family. 

As observed by Aistleitner, Mahatab, Munsch and Peyrot \cite{AMMP19}, adapting the long resonator method to the family of quadratic characters $\chi_d$, with $d$ varying over fundamental discriminants, appears to be a difficult problem. The main obstruction is that the orthogonality relations in this family are much subtler than in the family of all Dirichlet characters, while the positivity properties required by the long resonator method are no longer directly available. More recently, Darbar and Maiti \cite{DaMa25} succeeded in implementing a long resonator argument in this setting under GRH, proving analogous lower bounds in the larger family of quadratic characters attached to fundamental discriminants $|d|\leq x$.
 
For cubic characters at $s=1$, Darbar, David, Lalin and Lumley \cite{DDLL24}, using the ideas of Granville and Soundararajan \cite{GrSo03}, obtained under GRH the lower bound
$$
|L(1,\chi)|
\geq
e^\gamma\left(\log_2 x+\log_3 x-\log(2\log 3)+o(1)\right)
$$
for many cubic characters $\chi$ of prime conductor at most $x$.

The goal of this note is to prove unconditional lower bounds for primitive characters of any fixed order $g$. Our argument avoids the resonance method and is instead inspired by the conditional construction of Granville--Soundararajan \cite{GrSo03}, where the values of the character at the first few primes are prescribed. However, unlike in \cite{GrSo03}, our method is unconditional and uses conductors that are products of two primes. The main point is that, by considering characters of conductors $q_1q_2$, where $q_1$ and $q_2$ are primes congruent to $1\pmod g$ and satisfying certain conditions, we obtain a thin family with a useful ``almost-positivity'' property. More precisely, we will show that the average of $\re(\chi(n))$ over our family is bounded below by a small negative error, while it  equals $1$ for all integers whose prime factors are all small. This almost-positivity property is a substitute for the positivity input that is usually needed in the resonance method.
We now state our main results. 

\begin{thm}\label{Thm:Main1}
Let $g\geq 2$ be a fixed integer and $x$ be large. 
There exists a primitive character $\chi\bmod Q$ of order $g$ with $x/4\leq Q\leq x$ such that 
$$ \re(L(1, \chi))\geq e^{\gamma} \left(\log\log x+\log\log\log x-\log(2\log g)+o(1)\right).$$
\end{thm}
In particular, in the cases $g=2, 3$ we recover the same lower bounds as in the conditional results of Granville--Soundararajan \cite{GrSo03} for quadratic characters, and those of  
Darbar, David, Lalin and Lumley \cite{DDLL24} for cubic characters.

For fixed $g\geq 2$ and fixed $1/2<\sigma<1$, we prove the following result, which in particular recovers the author's conditional lower bounds \cite{La11} in the quadratic case.

\begin{thm}\label{Thm:Main2}
Let $1/2<\sigma<1$ be fixed. Let $g\geq 2$ be a fixed integer and $x$ be large. There exists a primitive character $\chi\bmod Q$ of order $g$ with $x/4\leq Q\leq x$ such that 
$$ \re (L(\sigma, \chi)) \geq \exp\left(\big(C_g(\sigma)+o(1)\big)\frac{(\log x)^{1-\sigma}}{(\log\log x)^{\sigma}}\right),$$
where $C_g(\sigma)
=
\frac{1}{1-\sigma}
\left(\frac{\sigma-1/2}{\sigma\log g}\right)^{1-\sigma}.$
\end{thm}

\subsection*{Acknowledgments} The author is supported by a junior chair of the Institut Universitaire de France. This work was completed while the author was visiting the Centre de Recherches Mathématiques in Montréal, which he thanks for its excellent working conditions.

\section{Proof of Theorems \ref{Thm:Main1} and \ref{Thm:Main2}}

\begin{lem}\label{Lem:ApproxL}
    Let $1/2\leq \sigma\leq 1$ be fixed. Let $\chi \pmod Q$ be a primitive Dirichlet character and $N\geq Q^{1/(2\sigma)}(\log Q)^{4}$ be a real number. Then we have 
$$ L(\sigma, \chi)=\sum_{n\leq N} \frac{\chi(n)}{n^{\sigma}}+ O\left(\frac{1}{\log Q}\right).
$$

\end{lem}
\begin{proof}
By partial summation and the  Pólya--Vinogradov inequality we have
\begin{align}\label{Eq:ApproxPV}
L(\sigma,\chi)&=\sum_{n\leq N}\frac{\chi(n)}{n^\sigma}+\sigma\int_N^\infty \frac{\sum_{n\leq t}\chi(n)}{t^{1+\sigma}}\,dt
-
\frac{1}{N^\sigma}\sum_{n\leq N}\chi(n)\nonumber\\
& =\sum_{n\leq N}\frac{\chi(n)}{n^\sigma}+ O\left(\frac{\sqrt{Q}\log Q}{N^{\sigma}}\right).
\end{align}
Using our assumption on $N$ completes the proof. 
\end{proof}

For each prime $q\equiv 1 \pmod g$, let $\ell_q$ be the smallest primitive root modulo $q$. We define $\psi_q$ to be the primitive character modulo $q$ such that 
\begin{equation}\label{DefCharRoot}
\psi_q(\ell_q)= e^{2\pi i/g}.
\end{equation}
Note that $\psi_q$ has order $g$. Furthermore, it follows from \cite[Lemma 2.2]{La17} that for distinct primes $q_1, q_2$ such that $q_1\equiv q_2\equiv 1\pmod g$, the character 
\begin{equation}\label{Eq:DefCharChi}
\chi_{q_1, q_2}:=\psi_{q_1}\overline{\psi_{q_2}}
\end{equation}
is primitive, of order $g$ and conductor $q_1q_2.$

For a positive integer $j$, we let $p_j$ be the $j$-th prime number. Let $\mu_g$ be the set of $g$-th roots of unity. Let $2\leq y\leq (\log x)^2$ be a real number and put $k=\pi(y)$. For any
$
\V=(v_1,\ldots,v_k) \in(\mu_g)^k,
$ let $\mathcal{P}_g(x, y; \V)$ be the set of primes $q\equiv 1 \pmod g$ with
$
x^{1/2}/2<q<x^{1/2}
$
and such that
$
\psi_q(p_j)=v_j
$
for all $j\leq k$. We now define the following family of characters
$$
\mathcal F_g(x, y)
=
\left\{
\chi_{q_1, q_2} : 
\ \frac{x^{1/2}}{2}<q_1<q_2<x^{1/2},
\ \text{and } (q_1,q_2)\in \p_g(x, y; \V)^2
\text{ for some } \V\in(\mu_g)^k
\right\}.
$$
Such a construction was introduced by the author in \cite{La17} to prove Littlewood-type lower bounds for $L(1,\chi)$ in families of primitive characters of fixed order. The main new observation here is that this family also carries a useful almost-positivity property. More precisely, the following lemma shows that for every positive integer $n$, the average of $\re(\chi(n))$ over $\chi\in\mathcal F_g(x,y)$ is bounded below by a small negative error.

\begin{lem}\label{Lem:Positivity}
Let $g\geq 2$ be a fixed integer, and $x$ be large. Let $y\geq 2$ be a real number such that $g^{\pi(y)+2}\leq c_0\sqrt{x}/\log x$ for some suitably small constant $c_0>0$. Then we have
\begin{equation}\label{Eq:SizeFamily}
|\mathcal F_g(x,y)|
\gg
\frac{x}{g^{\pi(y)+2}(\log x)^2},   
\end{equation}
and for all positive integers $n\geq 1$, we have
\begin{equation}\label{Eq:Positivity}
\frac{1}{|\mathcal F_g(x,y)|}
\sum_{\chi\in \mathcal F_g(x,y)} \re(\chi(n))
\geq
-c_1\frac{g^{\pi(y)+2}\log x}{\sqrt{x}},  
\end{equation}
for some absolute constant $c_1>0.$
\end{lem}
\begin{proof}
Let $k=\pi(y)$. We start by proving \eqref{Eq:SizeFamily}. By the definition of $\F_g(x, y)$ we have 
\begin{equation}\label{Eq:CalculSizeF}
|\mathcal F_g(x, y)|
=
\sum_{\V\in(\mu_g)^k}\binom{|\p_g(x, y; \V)|}{2}
=
\frac12\sum_{\V\in(\mu_g)^k}|\p_g(x, y; \V)|^2
+
O\left(\frac{\sqrt x}{\log x}\right),
\end{equation}
since
\begin{equation}\label{Eq:PrimesAP}
\sum_{\V\in(\mu_g)^k}|\p_g(x, y; \V)|
\sim
\frac{\sqrt x}{\phi(g)\log x},   
\end{equation}
by the prime number theorem in arithmetic progressions. 
By the Cauchy--Schwarz inequality and \eqref{Eq:PrimesAP} we obtain
$$
\sum_{\V\in(\mu_g)^k}|\p_g(x, y; \V)|^2
\geq
\frac{1}{g^k}
\left(\sum_{\V\in(\mu_g)^k}|\p_g(x, y; \V)|\right)^2
\gg
\frac{1}{g^{k+2}}\frac{x}{(\log x)^2}.
$$
Combining this bound with \eqref{Eq:CalculSizeF} completes the proof of \eqref{Eq:SizeFamily}.

We now prove 
\eqref{Eq:Positivity}. We observe that 
$$
\sum_{\chi\in \mathcal F_g(x,y)} \re \big(\chi(n)\big)
=
\sum_{\V\in(\mu_g)^k}
\sum_{\substack{\frac{\sqrt{x}}2<q_1<q_2<\sqrt{x}\\
(q_1,q_2)\in \p_g(x,y;\V)^2}}
\re\left(\psi_{q_1}(n)\overline{\psi_{q_2}(n)}\right)
$$
$$
=
\frac12
\sum_{\V\in(\mu_g)^k}
\Bigg|
\sum_{\substack{\frac{\sqrt{x}}2<q<\sqrt{x}\\
q\in \p_g(x,y;\V)}}
\psi_q(n)
\Bigg|^2
-
\frac{1}{2}\sum_{\substack{\frac{\sqrt{x}}2<q<\sqrt{x}\\
q\equiv 1 \bmod g}}|\psi_q(n)|^2.
$$
By the prime number theorem in arithmetic progressions and \eqref{Eq:SizeFamily}, we deduce that
$$
\frac{1}{|\mathcal F_g(x,y)|}
\sum_{\chi\in \mathcal F_g(x,y)} \re(\chi(n))
\geq -\frac{1}{2|\mathcal F_g(x,y)|}\sum_{\substack{\frac{\sqrt{x}}2<q<\sqrt{x}\\
q\equiv 1 \bmod g}}|\psi_q(n)|^2
\geq
-c_1\frac{g^{k+2}\log x}{\sqrt{x}}
$$
for some positive absolute constant $c_1$. This completes the proof.
\end{proof}

\begin{proof}[Proof of Theorems \ref{Thm:Main1} and \ref{Thm:Main2}]
We will prove both theorems at once, so we let $1/2<\sigma\leq 1$ be fixed. Define
$$
y:=\left(\frac{1-\frac{1}{2\sigma}}{\log g}\right)\log x\log_2 x\left(1-\frac{3}{\log_2x}\right),
$$ and let $k=\pi(y)$. 
Then by the prime number theorem, it follows that for $x$ sufficiently large we have 
\begin{align*}
k &\leq \frac{y}{\log y}\left(1+\frac{2}{\log y}\right)\leq \left(\frac{1-\frac{1}{2\sigma}}{\log g}\right)\log x\left(1-\frac{3}{\log_2x}\right)\left(1+\frac{2}{\log_2x}\right)\\
& \leq \left(\frac{1-\frac{1}{2\sigma}}{\log g}\right)\log x\left(1-\frac{1}{2\log_2x}\right),
\end{align*}
since $\log y\geq \log_2 x$. Therefore, we deduce that 
\begin{equation}\label{Eq:ConditionK}
g^{k+2}\leq \frac{x^{1-1/(2\sigma)}}{(\log x)^4}.
\end{equation}
In particular, since $1-1/(2\sigma)\leq 1/2$, the hypothesis of Lemma \ref{Lem:Positivity} is satisfied for $x$ large. Let $N= x^{1/(2\sigma)}(\log x)^4$. By Lemma \ref{Lem:ApproxL} we have
\begin{align}
\frac{1}{|\mathcal F_g(x,y)|}\sum_{\chi\in\mathcal F_g(x,y)} \re (L(\sigma,\chi))
&=
\frac{1}{|\mathcal F_g(x,y)|}\sum_{\chi\in\mathcal F_g(x,y)}
\left(
\re\sum_{n\leq N}\frac{\chi(n)}{n^\sigma}
+
O\left(\frac{1}{\log x}\right)
\right)\nonumber\\
&=
\sum_{n\leq N}\frac{1}{n^\sigma}
\frac{1}{|\mathcal F_g(x,y)|}\sum_{\chi\in\mathcal F_g(x,y)}\re(\chi(n))
+
O\left(\frac{1}{\log x}\right).
\end{align}
Let $P^+(n)$ be the largest prime factor of $n$ with the standard convention $P^+(1)=1$. If $P^+(n)\leq y$, then for all $\chi=
\chi_{q_1, q_2}\in \F_g(x, y)$ we have 
$
\chi_{q_1, q_2}(n)=\psi_{q_1}\overline{\psi_{q_2}}(n)=1.
$
Otherwise, we will use  \eqref{Eq:Positivity} to lower bound the sum over $\chi\in \F_g(x, y)$. This gives 
\begin{equation}\label{Eq:LowerBoundReL}
\frac{1}{|\mathcal F_g(x,y)|}\sum_{\chi\in\mathcal F_g(x,y)} \re (L(\sigma,\chi))
\geq
\sum_{\substack{n\leq N\\ P^+(n)\leq y}}
\frac{1}{n^\sigma}
+
O\left(g^{k+2}x^{1/(2\sigma)-1}(\log x)^3 + \frac{1}{\log x}\right).
\end{equation}
On the other hand, by a simple application of Rankin's trick  we get
\begin{align*}
\sum_{\substack{n>N\\  P^+(n)\leq y}}
\frac{1}{n^\sigma}
&\ll
\frac{1}{N^{1/4}}
\sum_{\substack{n\geq 1\\  P^+(n)\leq y}}
\frac{1}{n^{\sigma-1/4}}
\ll
\frac{1}{N^{1/4}}
\exp\left(O\left(\sum_{p\leq y}\frac{1}{p^{\sigma-1/4}}\right)\right)
\\
&\ll
\frac{1}{x^{1/8}}
\exp\left(y^{3/4}\right)\ll x^{-1/10}.
\end{align*}
Hence we deduce that 
$$
\sum_{\substack{n\leq N\\ P^+(n)\leq y}}
\frac{1}{n^\sigma}
=
\prod_{p\leq y}
\left(1-\frac{1}{p^\sigma}\right)^{-1}
+
O\left(x^{-1/10}\right).
$$
Inserting this estimate in \eqref{Eq:LowerBoundReL} and using \eqref{Eq:ConditionK} we obtain 
$$\max_{\chi\in \F_g(x, y)} \re(L(\sigma, \chi))\geq \frac{1}{|\mathcal F_g(x,y)|}\sum_{\chi\in\mathcal F_g(x,y)} \re (L(\sigma,\chi))\geq  \prod_{p\leq y}\left(1-\frac{1}{p^{\sigma}}\right)^{-1}+O\left(\frac{1}{\log x}\right).
$$
Finally, by the prime number theorem we have 
$$ 
\prod_{p\leq y}\left(1-\frac{1}{p^{\sigma}}\right)^{-1}= \begin{cases} e^{\gamma} \log y + O(1/\log y),
& \text{ if } \sigma=1,\\[1em]
\displaystyle
\exp\left(
\frac{y^{1-\sigma}}{(1-\sigma)\log y}
+O_{\sigma}\left(\frac{y^{1-\sigma}}{(\log y)^2}\right)
\right), & \text{ if } 1/2<\sigma<1.\end{cases}
$$
Moreover, with our choice of $y$ we have 
$$
\log y
=
\log_2 x+\log_3 x+\log\left(\frac{1-\frac{1}{2\sigma}}{\log g}\right)+o(1),
$$
and
$$
\frac{y^{1-\sigma}}{(1-\sigma)\log y}
=
\left(
\frac{1}{1-\sigma}
\left(\frac{1-\frac{1}{2\sigma}}{\log g}\right)^{1-\sigma}
+o(1)
\right)
\frac{(\log x)^{1-\sigma}}{(\log_2 x)^\sigma}.
$$
Combining the above estimates completes the proof.

\end{proof}

\end{document}